\documentclass[a4paper, 12point]{article}
\usepackage{}
\newcommand{\be}{\begin{equation}}
\newcommand{\ee}{\end{equation}}
\newcommand{\bea}{\begin{eqnarray}}
\newcommand{\eea}{\end{eqnarray}}

\usepackage{graphicx,epsfig}
\usepackage{enumerate}
\usepackage{amssymb}
\usepackage{mathrsfs}
\usepackage{amsthm}
\usepackage{amsmath}
\usepackage{setspace}
\usepackage{endnotes}
\newtheorem{thm}{Theorem}[section]
\theoremstyle{definition}

\theoremstyle{lemma}
\newtheorem{lem}{Lemma}[section]
\theoremstyle{example}

\theoremstyle{illustration}

\theoremstyle{proposition}
\newtheorem{pro}{Proposition}[section]
\theoremstyle{corollary}
\newtheorem{cor}{Corollary}[section]
\numberwithin{equation}{section}
\begin{document}
\date{}
\title{\textbf{Some Paranormed Difference Sequence Spaces Derived by Using Generalized Means}}
\author{Atanu Manna\footnote{Corresponding author e-mail: atanumanna@maths.iitkgp.ernet.in}, Amit Maji\footnote{Author's e-mail: amit.iitm07@gmail.com}, P. D. Srivastava \footnote{Author's e-mail: pds@maths.iitkgp.ernet.in}\\
\textit{\small{Department of Mathematics, Indian Institute of Technology, Kharagpur}} \\
\textit{\small{Kharagpur 721 302, West Bengal, India}}}
\maketitle
\vspace{20pt}
\begin{center}\textbf{Abstract}\end{center}
This paper presents new sequence spaces $X(r, s, t, p ;\Delta)$ for
$X \in \{ l_\infty(p), c(p), c_0(p), l(p)\}$ defined by using
generalized means and difference operator. It is shown that these
spaces are complete under a suitable paranorm. Furthermore, the
 $\alpha$-, $\beta$-, $\gamma$- duals of these sequence spaces are computed
 and also obtained necessary and sufficient conditions for some matrix transformations from $X(r, s, t, p ;\Delta)$ to $X$. Finally,
 it is proved that the sequence space $l(r, s, t, p ;\Delta)$ is rotund when $p_n>1$ for all $n$ and has the Kadec-Klee property.
\section{Introduction}
The study of sequence spaces play an important role in several branches of analysis, namely, the structural theory of topological vector spaces, summability theory, Schauder basis theory etc. Besides this, the theory of sequence spaces is a powerful tool for obtaining some topological and geometrical results with the help of Schauder basis.

Let $w$ be the space of all real or complex sequences $x=(x_n)$, $n\in \mathbb{N}_0$. For an infinite matrix $A$ and a sequence space $\lambda$, the matrix domain of $A$, which is denoted by $\lambda_{A}$ and defined as $\lambda_A=\{x\in w: Ax\in \lambda\}$ \cite{WIL}. Basic methods, which are used to determine the topologies, matrix transformations and inclusion relations on sequence spaces can also be applied to study the matrix domain $\lambda_A$. Recently, there is an approach of forming new sequence spaces by using matrix domain of a suitable matrix and characterize the matrix mappings between these sequence spaces.

Let $(p_k)$ be a bounded sequence of strictly positive real numbers such that $H=\displaystyle\sup_k p_k$ and $M=\max \{1, H\}$. The linear spaces $l_\infty(p)$, $c(p), c_0(p)$ and $l(p)$ are introduced and studied by Maddox \cite{MAD}, where
\begin{align*}
&l_\infty(p)=\Big\{x=(x_k)\in w:  \displaystyle \sup_{k}|x_k|^{p_k}<\infty\Big\},\\
&c(p)=\Big\{x=(x_k)\in w:  \displaystyle \lim_{k\rightarrow\infty}|x_k-l|^{p_k}=0 \mbox{~~for some scalar~} l\Big\}\mbox{~and~},\\
&c_0(p)=\Big\{x=(x_k)\in w:  \displaystyle \lim_{k\rightarrow\infty}|x_k|^{p_k}=0 \Big\},\\
&l(p)=\Big\{x=(x_k)\in w:  \displaystyle \sum_{k=0}^{\infty}|x_k|^{p_k}<\infty\Big\}.
\end{align*}
The linear spaces $l_\infty(p)$, $c(p), c_0(p)$ are complete with the paranorm $g(x) = \displaystyle \sup_{k}|x_k|^{p_k \over M}$ if and only if $\inf p_k>0$ for all $k$ while $l(p)$ is complete with the paranorm $\tilde g(x) =
\displaystyle \Big(\sum_{k}|x_k|^{p_k }\Big)^{1 \over M}$.
Recently, several authors introduced new sequence spaces by using matrix domain. For example, Ba\c{s}ar et al. \cite{BAS2} studied the space $bs(p)=[l_\infty(p)]_S$, where $S$ is the summation matrix. Altay and Ba\c{s}ar \cite{ALT1} studied the sequence spaces $r^t(p)$ and $r_{\infty}^t(p)$, which consist of all sequences whose Riesz transform are in the spaces $l(p)$ and $l_\infty(p)$ respectively, i.e., $r^t(p)=[l(p)]_{R^t}$ and $r_{\infty}^t(p)= [l_{\infty}(p)]_{R^t}$. Altay and Ba\c{s}ar  also studied the sequence spaces $r_{c}^t(p)= [c(p)]_{R^t}$ and $r_{0}^t(p)= [c_0(p)]_{R^t}$ in \cite{ALT}.

Kizmaz first introduced and studied the difference sequence space in \cite{KIZ}.
Later on, many authors including Ahmad and Mursaleen \cite{AHM}, \c{C}olak and Et \cite{COL}, Ba\c{s}ar and Altay\cite{ALT} etc.
studied new sequence spaces defined by using difference operator. Using Euler and difference operator, Karakaya and Polat
introduced the paranormed sequence spaces $e_0^\alpha(p; \Delta),e^\alpha_c(p; \Delta)$ and $e_{\infty}^\alpha(p; \Delta)$ in \cite{KAR}.
Mursaleen and Noman \cite{MUR} introduced a sequence
space of generalized means, which includes most of the earlier known
sequence spaces.

In $2012$, Demiriz and \c{C}akan \cite{DEM} introduced the new paranormed difference sequence spaces $\lambda(u, v, p; \Delta)$ for $\lambda \in \{l_\infty(p), c(p), c_0(p), l(p)\}$, combining weighted mean and difference operator, defined as
$$\lambda(u, v, p; \Delta)=\Big\{x\in w: (G(u, v).\Delta) x \in \lambda\Big\},$$
where the matrices $G(u, v)=(g_{nk})$ and $\Delta=(\delta_{nk})$ are given by
\begin{align*}
g_{nk} &= \left\{
\begin{array}{ll}
    u_nv_k & \quad \mbox{~if~} 0\leq k \leq n,\\
    0 & \quad \mbox{~if~} k > n
\end{array}\right.& \mbox{and~}
\delta_{nk}& = \left\{
\begin{array}{ll}
   0 & \quad \mbox{~if~} 0\leq k <n-1,\\
    (-1)^{n-k} & \quad \mbox{~if~} n-1\leq k \leq n,\\
    0 & \quad \mbox{~if~}  k>n.
\end{array}\right.
\end{align*}
By using matrix domain, one can write $c_0(u, v, p; \Delta)=[c_0(p)]_{G(u, v;\Delta)}$, $c(u, v, p; \Delta)=[c(p)]_{G(u, v;\Delta)}$, $l_{\infty}(u, v, p; \Delta)=[l_{\infty}(p)]_{G(u, v;\Delta)}$ and $l(u, v, p; \Delta)=[l{(p)}]_{G(u, v;\Delta)}$.

The aim of this present paper is to introduce and study new sequence spaces $X(r,s,t,p; \Delta)$ for $X \in \{l_\infty(p), c(p), c_0(p), $ $l(p)\}$. It is shown that these spaces are complete paranormed sequence spaces under some suitable paranorm. Some topological results and the $\alpha$-, $\beta$-, $\gamma$- duals of these spaces are obtained. A characterization of some matrix transformations between these new sequence spaces is established. It is also shown that the sequence space $l(r, s, t, p ;\Delta)$ is rotund when $p_n>1$ for all $n$ and has the Kadec-Klee property.

\section{Preliminaries}
 Let $l_\infty, c$ and $c_0$ be the spaces of all bounded, convergent and null sequences $x=(x_n)$ respectively, with norm $\|x\|_\infty=\displaystyle\sup_{n}|x_n|$. Let  $bs$ and $cs$ be the sequence spaces of all bounded and convergent series respectively. We denote by $e=(1, 1, \cdots)$ and $e_{n}$ for the sequence whose $n$-th term is $1$ and others are zero and $\mathbb{{N_{\rm 0}}}=\mathbb{N}\cup \{0\}$, where $\mathbb{N}$ is the set of all natural numbers.\\
For any subsets $U$ and $V$ of $w$, the multiplier space $M(U, V)$ of $U$ and $V$ is defined as
\begin{center}
$M(U, V)=\{a=(a_n)\in w : au=(a_nu_n)\in V ~\mbox{for all}~ u\in U\}$.
\end{center}
In particular,
\begin{center}
$U^\alpha= M(U, l_1)$, $U^\beta= M(U, cs)$ and $U^\gamma= M(U, bs)$
\end{center} are called the $\alpha$-, $\beta$- and $\gamma$- duals of $U$ respectively \cite{WIL}.

Let $A=(a_{nk})_{n, k}$ be an infinite matrix
with real or complex entries $a_{nk}$. We write $A_n$ as the
sequence of the $n$-th row of $A$, i.e.,
$A_n=(a_{nk})_{k}$ for every $n$.
For $x=(x_n)\in w$, the $A$-transform of $x$ is defined as the
sequence $Ax=((Ax)_n)$, where
\begin{center}
$A_n(x)=(Ax)_n=\displaystyle\sum_{k=0}^{\infty}a_{nk}x_k$,
\end{center}
provided the series on the right side converges for each $n$. For any two sequence spaces $U$ and $V$, we denote by $(U, V)$, the class of all infinite matrices $A$ that map $U$ into $V$. Therefore $A\in (U, V)$ if and only if $Ax=((Ax)_n)\in V$ for all $x\in U$. In other words, $A\in (U, V)$ if and only if $A_n \in U^\beta$ for all $n$ \cite{WIL}.
An infinite matrix $T={(t_{nk})}_{n,k}$ is said to be triangle if $t_{nk}=0$ for $k>n$ and $t_{nn}\neq 0$, $n\in \mathbb{{N_{\rm 0}}}$.

\section{Sequence spaces $X(r, s, t, p; \Delta)$ for $X \in \{ l_{\infty}(p), c(p), c_{0}(p), l(p)\}$.}
In this section, we first begin with the notion of generalized means given by Mursaleen et al. \cite{MUR}.\\
We denote the sets $\mathcal{U}$ and $\mathcal{U}_{0}$ as
\begin{center}
$ \mathcal{U} = \Big \{ u =(u_{n})_{n=0}^{\infty} \in w: u_{n} \neq
0~~ {\rm for~ all}~~ n \Big \}$ and $ \mathcal{U_{\rm 0}} = \Big \{ u
=(u_{n})_{n=0}^{\infty} \in w: u_{0} \neq 0 \Big \}.$
\end{center}
Let $r, t \in \mathcal{U}$ and $s \in \mathcal{U}_{0}$. The sequence $y=(y_{n})$ of generalized means of a sequence $x=(x_{n})$ is defined
by $$ y_{n}= \frac{1}{r_{n}}\sum_{k=0}^{n} s_{n-k}t_{k}x_{k} \qquad (n \in \mathbb{N_{\rm 0}}).$$
The infinite matrix $A(r, s, t)$ of generalized means is defined by

$$(A(r,s,t))_{nk} = \left\{
\begin{array}{ll}
    \frac{s_{n-k}t_{k}}{r_{n}} & \quad 0\leq k \leq n,\\
    0 & \quad k > n.
\end{array}\right. $$

Since $A(r, s, t)$ is a triangle, it has a unique inverse and the
inverse is also a triangle \cite{JAR}. Take $D_{0}^{(s)} =
\frac{1}{s_{0}}$ and

$ D_{n}^{(s)} =
\frac{1}{s_{0}^{n+1}} \left|
\begin{matrix}
    s_{1} & s_{0} &  0 & 0 \cdots & 0 \\
    s_{2} & s_{1} & s_{0}& 0 \cdots & 0 \\
    \vdots & \vdots & \vdots & \vdots    \\
    s_{n-1} & s_{n-2} & s_{n-3}& s_{n-4} \cdots & s_0 \\
      s_{n} & s_{n-1} & s_{n-2}& s_{n-3} \cdots & s_1
\end{matrix} \right| \qquad \mbox{for} $ $n\geq 1.$\\ \\

Then the inverse of $A(r, s, t)$ is the triangle $B= (b_{nk})_{n, k}$ which is defined as
$$b_{nk} = \left\{
\begin{array}{ll}
    (-1)^{n-k}~\frac{D_{n-k}^{(s)}}{t_{n}}r_{k} & \quad 0\leq k \leq n,\\
    0 & \quad k > n.
\end{array}\right. $$
Throughout this paper, we consider $p=(p_k)$ be a bounded sequence of strictly positive real numbers such that $H=\displaystyle\sup_kp_k$ and $M=\max\{1, H\}$. \\
We now introduce the sequence spaces $X(r, s, t, p; \Delta)$ for $X \in \{ l_{\infty}(p), c(p), c_{0}(p), l(p)\}$ as
$$ X(r, s, t, p; \Delta)= \bigg \{ x=(x_{k})\in w : \Big( \frac{1}{r_{n}}\sum_{k=0}^{n}s_{n-k}t_{k} \Delta {x_{k}}\Big)_{n} \in X \bigg  \},$$
which is a combination of generalized means and difference operator $\Delta$ such that $\Delta x_k =x_k -x_{k-1}$, $x_{-1}=0$. By using matrix domain, we can write $X(r, s,t, p; \Delta)=  X_{A(r, s,t; \Delta)}=\{x \in w :A(r, s,t; \Delta)x\in X\}$, where $A(r, s,t; \Delta)= A(r, s,t). \Delta$, product of two triangles $A(r, s,t)$ and $\Delta$. These sequence spaces include many well known sequence spaces studied by several authors as follows:
\begin{enumerate}[I.]
\item if $r_{n}=\frac{1}{u_{n}}$, $t_{n}=v_{n}$ and $s_{n}=1~ \forall~ n$, then the sequence spaces $X(r, s,t, p; \Delta)$ reduce to $ X(u, v, p ; \Delta)$ for $X \in \{ l_{\infty}(p), c(p), c_{0}(p), l(p)\}$  introduced and studied by Demiriz and \c{C}akan \cite{DEM}.
\item  if $t_{n}=1$, $s_{n}=1$ ~$\forall~ n$ and $r_{n}=n+1$, then the sequence space $ l(r, s,t, p; \Delta)$ reduces to the non absolute type sequence space $ X_p(\Delta)$ studied by Ba\c{s}arir \cite{BAS1}.
\item if $r_{n}= \frac{1}{n!}, $ $t_{n}=\frac{\alpha^n}{n!}$, $s_{n}=\frac{(1-\alpha)^n}{n!}$, where $0<\alpha<1$, then the sequence spaces $ X(r, s,t, p; \Delta)$ for $X \in \{ l_{\infty}(p), c(p), c_{0}(p)\}$ reduce to $e^\alpha_\infty(p; \Delta)$, $e^\alpha_c(p; \Delta)$ and $e^\alpha_0(p; \Delta)$ respectively introduced and studied by Karakaya and Polat \cite{KAR}.
\item if $ r_{n}=n+1,$ $t_{n}={1+ \alpha^n}$, $0<\alpha<1$ and $s_{n}=1$, $p_n=1$ $\forall n$, then the sequence spaces $ c(r, s,t, p ; \Delta)$ and $ c_0(r, s,t, p ; \Delta)$ reduce to the sequence spaces $a_c^\alpha(\Delta)$ and $a_0^\alpha(\Delta)$ respectively studied by Demiriz and \c{C}akan \cite{DEM1}.
\end{enumerate}

\section{Main results}
Throughout the paper, we denote the sequence spaces $X(r, s, t, p; \Delta)$ as $l(r, s, t, p; \Delta)$, $c_0(r, s, t, p; \Delta)$, $c(r, s, t, p; \Delta)$ and $l_{\infty}(r, s, t, p; \Delta)$
for $X =l(p),c_0(p), c(p)$ and $l_{\infty}(p)$ respectively.
\begin{thm}
$(a)$ The sequence space $l(r, s, t, p; \Delta)$ is a complete
linear metric space paranormed by $\tilde{h}$ defined as
\begin{center}
$\tilde{h}(x)=\Big(\displaystyle\sum_{n=0}^{\infty}\Big|\frac{1}{r_n}\displaystyle\sum_{k=0}^{n}s_{n-k}t_k\Delta
x_k\Big|^{p_n}\Big)^\frac{1}{M}$.
\end{center}
$(b)$ The sequence spaces $X(r,s, t, p ; \Delta)$ for $X\in
\{l_{\infty}(p)$, $ c(p), c_{0}(p)\}$ are complete linear metric
spaces paranormed by $h$ defined as
\begin{center}
$h(x)=\displaystyle\sup_{n}\Big|\frac{1}{r_n}\displaystyle\sum_{k=0}^{n}s_{n-k}t_k\Delta x_k\Big|^\frac{p_n}{M}$.
\end{center}
\end{thm}

\begin{proof}We prove the part $(a)$ of this theorem. In
a similar way, we can prove the part $(b)$. \\Let
$x$, $y\in l(r, s, t, p; \Delta)$. Using Minkowski's inequality
\begin{align}
\bigg(\displaystyle\sum_{n=0}^{\infty}\Big|\frac{1}{r_n}\displaystyle\sum_{k=0}^{n}s_{n-k}t_k\Delta (x_k+ y_k)\Big|^{p_n}\bigg)^\frac{1}{M} & \leq
\bigg(\displaystyle\sum_{n=0}^{\infty}\Big|\frac{1}{r_n}\displaystyle\sum_{k=0}^{n}s_{n-k}t_k\Delta x_k\Big|^{p_n}\bigg)^\frac{1}{M}\nonumber\\  & ~~+ \bigg(\displaystyle\sum_{n=0}^{\infty}\Big|\frac{1}{r_n}\displaystyle\sum_{k=0}^{n}s_{n-k}t_k\Delta y_k\Big|^{p_n}\bigg)^\frac{1}{M}<\infty,
\end{align}so we have $x+y \in l(r, s, t, p; \Delta)$.\\
Let $\alpha$ be any scalar. Since $|\alpha|^{p_k}\leq \max\{1, |\alpha|^M\}$ for any scalar $\alpha$, we have
 $\tilde{h}(\alpha x)\leq \max\{1, |\alpha|\}\tilde{h}(x)$. Hence $\alpha x \in l(r, s, t, p; \Delta)$.
It is trivial to show that $\tilde{h}(\theta)=0$, $\tilde{h}(-x)=\tilde{h}(x)$ for all $x\in l(r, s, t, p; \Delta)$
 and subadditivity of $\tilde{h}$, i.e., $\tilde{h}(x+y)\leq \tilde{h}(x)+ \tilde{h}(y)$ follows from $(4.1)$.\\
 Next we show that the scalar multiplication is continuous.
Let $(x^m)$ be a sequence in $l(r, s,
t, p; \Delta)$, where $x^{m}= (x_k^{m})=(x_0^{m}, x_1^{m}, x_2^{m}, \ldots )$ $\in l(r, s, t, p; \Delta)$ for each $m\in \mathbb{N}_0$ such that $\tilde{h}(x^{m}-x)\rightarrow0$ as
$m\rightarrow\infty$ and $(\alpha_m)$ be a sequence of scalars such
that $\alpha_m\rightarrow \alpha$ as $m\rightarrow\infty$. Then
$\tilde{h}(x^{m})$ is bounded that follows from the following
inequality
\begin{center}
$\tilde{h}(x^{m})\leq \tilde{h}(x)+ \tilde{h}(x-x^{m})$.
\end{center}
Now consider
\begin{align*}
\tilde{h}(\alpha_mx^{m}-\alpha x) & = \bigg(\displaystyle\sum_{n=0}^{\infty}\Big|\frac{1}{r_n}\displaystyle\sum_{k=0}^{n}s_{n-k}t_k\Delta (\alpha_mx^{m}_k-\alpha x_k)\Big|^{p_n}\bigg)^\frac{1}{M} \\ &=
\bigg(\displaystyle\sum_{n=0}^{\infty}\Big|\frac{1}{r_n}\displaystyle\sum_{k=0}^{n}s_{n-k}t_k\Delta \big((\alpha_m-\alpha)x^{m}_k + \alpha (x^{m}_k- x_k)\big)\Big|^{p_n}\bigg)^\frac{1}{M}\nonumber\\  & \leq |\alpha_m-\alpha| \tilde{h}(x^{m}) + |\alpha| \tilde{h}(x^{m}-x)\rightarrow 0 \mbox{~as~} m\rightarrow\infty.
\end{align*}This shows that the scalar multiplication is continuous. Hence $\tilde{h}$ is a paranorm on the space $l(r, s, t, p; \Delta)$.\\
Now we show that the completeness of the space $l(r, s, t, p;
\Delta)$ with respect to the paranorm $\tilde{h}$. Let
$(x^{m})$ be a Cauchy sequence in $l(r,
s, t, p; \Delta)$. So for every $\epsilon>0$ there is a
$n_0 \in \mathbb{N}$ such that
\begin{center}
$\tilde{h}(x^{m}-x^{l})< \frac{\epsilon}{2}$ ~ for all $m, l\geq
n_0$.
\end{center}
Then by definition of $\tilde{h}$, we have for each $n$
\begin{equation}
\displaystyle \bigg|(A(r, s, t; \Delta)x^{m})_n -(A(r, s, t; \Delta)x^{l})_n \bigg| \leq
 \bigg(\displaystyle\sum_{n=0}^{\infty}\bigg|(A(r, s, t; \Delta)x^{m})_n -(A(r, s, t; \Delta)x^{l})_n \bigg|^{p_n} \bigg)^{1 \over M} <
\frac{\epsilon}{2}
\end{equation}
~\mbox{for all}~ $m, l\geq n_0,$ which implies that the
sequence $((A(r, s, t; \Delta)x^{m})_n)$ is a
Cauchy sequence of scalars for each fixed $n$
and hence converges for each $n$. We write
$$ \displaystyle\lim_{m \rightarrow \infty} (A(r, s, t; \Delta)x^{m})_n = (A(r, s, t; \Delta)x)_n \quad ( n \in \mathbb{N_{\rm 0}}).$$
Now taking $l \rightarrow \infty$ in (4.2), we obtain
$$\bigg(\displaystyle\sum_{n=0}^{\infty}\bigg|(A(r, s, t; \Delta)x^{m})_n -(A(r, s, t; \Delta)x)_n \bigg|^{p_n} \bigg)^{1 \over M} <
\epsilon$$
~\mbox{for all}~ $m \geq n_0$ and each fixed $n$. Thus $(x^{m})$ converges to $x$ in $l(r, s, t, p; \Delta)$
with respect to $\tilde{h}$.\\
 To show  $x\in l(r, s, t, p;
\Delta)$, we take
\begin{align*}
\Big(\displaystyle\sum_{n=0}^{\infty}\Big|\frac{1}{r_n}\displaystyle\sum_{k=0}^{n}s_{n-k}t_k\Delta x_k\Big|^{p_n}\Big)^\frac{1}{M}
&= \Big(\displaystyle\sum_{n=0}^{\infty}\Big|\frac{1}{r_n}\displaystyle\sum_{k=0}^{n}s_{n-k}t_k\Delta (x_k -x_{k}^{m} + x_{k}^{m}) \Big|^{p_n}\Big)^\frac{1}{M}\\
& \leq \Big(\displaystyle\sum_{n=0}^{\infty}\Big|\frac{1}{r_n}\displaystyle\sum_{k=0}^{n}s_{n-k}t_k\Delta (x_k - x_{k}^{m})\Big|^{p_n}\Big)^\frac{1}{M}\\
& ~~+ \Big(\displaystyle\sum_{n=0}^{\infty}\Big|\frac{1}{r_n}\displaystyle\sum_{k=0}^{n}s_{n-k}t_k\Delta x_k^{m}\Big|^{p_n}\Big)^\frac{1}{M}\\
&= \tilde{h}(x -x^{m}) + \tilde{h}(x^{m}) < \infty \quad
\mbox{for all}~ m \geq n_0.
\end{align*}Therefore $x\in l(r, s, t, p; \Delta)$. This completes
the proof.
\end{proof}

\begin{thm}
The sequence spaces $X(r,s,t, p; \Delta)$ for $X \in \{
l_{\infty}(p)$, $ c(p), c_{0}(p), l(p)\}$ are linearly isomorphic to
the spaces $X \in \{ l_{\infty}(p)$, $ c(p), c_{0}(p), l(p)\}$
respectively, i.e., $l_{\infty}(r, s, t, p ; \Delta) \cong
l_{\infty}(p)$,  $c(r, s, t, p ; \Delta) \cong c(p)$, $c_{0}(r, s,
t, p ; \Delta) \cong c_{0}(p)$ and $l(r,s,t, p; \Delta) \cong l(p)$.
\end{thm}
\begin{proof}
We prove the theorem only for the case when $X=l(p)$. For
this, we need to show that there exists a bijective linear map from
$l(r, s, t, p ;\Delta)$ to $l(p)$. Now we define a map $T:l(r, s, t,
p ; \Delta) \rightarrow l(p)$ by $x \mapsto Tx = y=(y_n)$, where
\begin{center}\label{k1}
$\displaystyle y_n= \frac{1}{r_{n}}\sum_{k=0}^{n}s_{n-k}t_{k} \Delta {x_{k}}$.
\end{center}
Since $\Delta$ is a linear operator, so the linearity of $T$ is
trivial. It is easy to see that $Tx=0$ implies $x=0$. Thus $T$ is
injective. To prove $T$ is surjective, let $y \in l(p)$. Since $y= (A(r, s, t). \Delta)x$, i.e.,
$$x= \big(A(r, s, t). \Delta\big)^{-1}y=\Delta^{-1}.A(r, s, t)^{-1}y, $$
we can get a sequence $x=(x_{n})$ as
\begin{equation}
 x_n = \sum_{j=0}^{n}\sum_{k=0}^{n-j}
(-1)^{k} \frac{D_{k}^{(s)}}{t_{k+j}} r_{j}y_{j} \qquad (n \in
\mathbb{N}_0).
\end{equation}
 Then
\begin{center}
$\tilde{h}(x)=\bigg(\displaystyle\sum_{n=0}^{\infty}\Big|\frac{1}{r_n}\displaystyle\sum_{k=0}^{n}s_{n-k}t_k\Delta x_k\Big|^{p_n}\bigg)^\frac{1}{M}=\Big(\displaystyle\sum_{n=0}^{\infty}\big|y_n\big|^{p_n}\Big)^\frac{1}{M}=\tilde{g}(y)<\infty$.
\end{center}
Thus $x \in l(r, s, t, p; \Delta)$ and this shows that $T$ is
surjective. Hence $T$ is a linear bijection from $l(r, s, t, p;
\Delta)$ to $l(p)$. Also $T$ is paranorm preserving. This completes
the proof.
\end{proof}

\subsection{The $\alpha$-, $\beta $-, $\gamma$-duals of $X(r,s, t, p; \Delta)$ for
$X\in\{l_\infty(p), c(p), c_0(p), l(p)\}$}
In $1999$, K. G. Grosse-Erdmann \cite{GRO} has characterized the matrix transformations between the sequence spaces of Maddox, namely, $l_\infty(p), c(p), c_0(p)$ and $ l(p)$.
To compute the $\alpha$-, $\beta $-, $\gamma$-duals of $X(r,s, t, p; \Delta)$ for
$X\in\{l_\infty(p), c(p), c_0(p), l(p)\}$ and to characterize the classes of matrix mappings between these spaces, we list the following conditions.\\

Let $L$, $N$ be any two natural numbers, $F$ denotes finite subset of $\mathbb{N}_0$ and $\alpha, \alpha_k$ are complex numbers. Let $p=(p_k)$, $q=(q_k)$ be bounded sequences of strictly positive real numbers and $A=(a_{nk})_{n, k}$ be an infinite matrix. We put $K_1=\{k\in \mathbb{N}_0: p_k\leq 1\}$
and $K_2=\{k\in \mathbb{N}_0: p_k> 1\}$ and $p_k'= \frac{p_k}{p_k-1}$ for $k\in K_2$.

\begin{align}
& \displaystyle\sup_{F}\displaystyle\sup_{k\in K_1}\Big|\sum_{n\in F}a_{nk}\Big|^{p_k}<\infty\\
& \exists L~ \displaystyle\sup_{F}\displaystyle\sum_{k\in K_2}\Big|\sum_{n\in F}a_{nk}L^{-1}\Big|^{p_k'}<\infty\\
& \displaystyle\lim_{n} |a_{nk}|^{q_n}=0 \mbox{~for all~} k\\
& \forall L~ \displaystyle\sup_{n}\displaystyle\sup_{k\in K_1}\Big|a_{nk}L^{\frac{1}{q_n}}\Big|^{p_k}<\infty\\
& \forall L~ \exists N~ \displaystyle\sup_{n}\displaystyle\sum_{k\in K_2}\Big|a_{nk}L^{\frac{1}{q_n}}N^{-1}\Big|^{p_k'}<\infty\\
& \displaystyle\sup_{n}\displaystyle\sup_{k\in K_1}|a_{nk}|^{p_k}<\infty\\
& \exists N \displaystyle\sup_{n}\displaystyle\sum_{k\in K_2}\Big|a_{nk}N^{-1}\Big|^{p_k'}<\infty\\
& \exists (\alpha_k)~ \displaystyle\lim_{n}|a_{nk}-\alpha_k|^{q_n}=0 \mbox{~for all~} k\\
& \exists (\alpha_k)~\forall L~ \displaystyle\sup_{n}\displaystyle\sup_{k\in K_1}\Big(|a_{nk}-\alpha_k|L^{\frac{1}{q_n}}\Big)^{p_k}<\infty\\
& \exists (\alpha_k)~\forall L~\exists N~ \displaystyle\sup_{n}\displaystyle\sum_{k\in K_2}\big(|a_{nk}-\alpha_k|L^{\frac{1}{q_n}}N^{-1}\big)^{p_k'}<\infty\\
& \exists L~ \displaystyle\sup_{n}\displaystyle\sup_{k\in K_1}\Big|a_{nk}L^{-\frac{1}{q_n}}\Big|^{p_k}<\infty\\
& \exists L~ \displaystyle\sup_{n}\displaystyle\sum_{k\in K_2}\Big|a_{nk}L^{-\frac{1}{q_n}}\Big|^{p_k'}<\infty\\
& \exists N~ \displaystyle\sup_{F}\displaystyle\sum_{n}\Big|\sum_{k\in F}a_{nk}N^{-\frac{1}{p_k}}\Big|<\infty\\
& \forall L~ \exists N~ \displaystyle\sup_{n} L^{\frac{1}{q_n}} \displaystyle\sum_{k} |a_{nk}|N^{-\frac{1}{p_k}}<\infty\\
& \exists N~ \displaystyle\sup_{n}\displaystyle\sum_{k}|a_{nk}|N^{-\frac{1}{p_k}}<\infty\\
& \exists (\alpha_k)~\forall L~\exists N~ \displaystyle\sup_{n} L^{\frac{1}{q_n}}\displaystyle\sum_{k} |a_{nk}-\alpha_k|N^{-\frac{1}{p_k}}<\infty
\end{align}
\begin{align}
& \exists N~ \displaystyle\sup_{n}\bigg(\displaystyle\sum_{k}|a_{nk}|N^{-\frac{1}{p_k}}\bigg)^{q_n}<\infty\\
& \displaystyle\sum_{n}\Big|\sum_{k}a_{nk}\Big|<\infty\\
& \displaystyle\lim_{n}\Big|\sum_{k}a_{nk}\Big|^{q_n}=0\\
& \exists \alpha~ \displaystyle\lim_{n}\Big|\sum_{k}a_{nk}-\alpha\Big|^{q_n}=0\\
& \displaystyle\sup_{n} \Big|\sum_{k}a_{nk}\Big|^{q_n}<\infty\\
& \forall N~ \displaystyle\sup_{F}\displaystyle\sum_{n} \Big|\sum_{k\in F}a_{nk}N^{\frac{1}{p_k}}\Big|<\infty\\
& \forall N~ \displaystyle\lim_{n}\bigg(\displaystyle\sum_{k}|a_{nk}|N^{\frac{1}{p_k}}\bigg)^{q_n}=0\\
& \forall N~ \displaystyle\sup_{n}\displaystyle\sum_{k} |a_{nk}|N^{\frac{1}{p_k}}<\infty\\
&  \exists (\alpha_k)~ \forall N~ \displaystyle\lim_{n} \bigg(\displaystyle\sum_{k}|a_{nk}-\alpha_k|N^{\frac{1}{p_k}}\bigg)^{q_n}=0\\
& \forall N~ \displaystyle\sup_{n}\bigg(\displaystyle\sum_{k}|a_{nk}|N^{\frac{1}{p_k}}\bigg)^{q_n}<\infty.
\end{align}
\begin{lem}\cite{GRO} {\label{lem1}}
$(i)$ $A \in (l(p), l_1)$ if and only if $(4.4)$ and $(4.5)$ hold.\\
$(ii)$ $A\in (l(p), c_0(q))$ if and only if $(4.6)$, $(4.7)$ and $(4.8)$ hold.\\
$(iii)$ $A\in (l(p), c(q))$ if and only if $(4.9)$, $(4.10)$, $(4.11)$,  $(4.12)$ and $(4.13)$ hold.\\
$(iv)$ $A\in (l(p), l_{\infty}(q))$ if and only if $(4.14)$ and $(4.15)$ hold.\\
\end{lem}
\begin{lem}\cite{GRO}{\label{lem2}}
$(i)$ $A \in (c_0(p), l_1)$ if and only if $(4.16)$ holds.\\
$(ii)$ $A\in (c_0(p), c_0(q))$ if and only if $(4.6)$ and $(4.17)$ hold.\\
$(iii)$ $A\in (c_0(p), c(q))$ if and only if $(4.11)$, $(4.18)$ and $(4.19)$ hold.\\
$(iv)$ $A\in (c_0(p), l_{\infty}(q))$ if and only if $(4.20)$ holds.\\
\end{lem}

\begin{lem}\cite{GRO}{\label{lem3}}
$(i)$ $A \in (c(p), l_1)$ if and only if $(4.16)$ and $(4.21)$ hold.\\
$(ii)$ $A\in (c(p), c_0(q))$ if and only if $(4.6)$ and $(4.17)$ and $(4.22)$ hold.\\
$(iii)$ $A\in (c(p), c(q))$ if and only if $(4.11)$, $(4.18)$, $(4.19)$ and $(4.23)$ hold.\\
$(iv)$ $A\in (c(p), l_{\infty}(q))$ if and only if $(4.20)$ and $(4.24)$ hold.\\
\end{lem}
\begin{lem}\cite{GRO}{\label{lem4}}
 $(i)$ $A \in (l_{\infty}(p), l_1)$ if and only if $(4.25)$ holds.\\
$(ii)$ $A\in (l_{\infty}(p), c_0(q))$ if and only if $(4.26)$ holds.\\
$(iii)$ $A\in (l_{\infty}(p), c(q))$ if and only if $(4.27)$ and $(4.28)$ hold.\\
$(iv)$ $A\in (l_{\infty}(p), l_{\infty}(q))$ if and only if $(4.29)$ holds.\\
\end{lem}
We consider the following sets to obtain $\alpha$-duals of the spaces $X(r,s,t, p ; \Delta)$.
\begin{align*}
&H_1(p)=\bigcup_{L\in \mathbb{N}}\Big\{a=(a_n)\in w:  \displaystyle\sup_{F}\displaystyle\sum_{n=0}^{\infty}\Big|\displaystyle\sum_{k\in F} \displaystyle\sum_{j=0}^{n-k} (-1)^j \frac{D_j^{(s)}}{t_{j+k}}r_ka_n L^{\frac{-1}{p_k}}\Big|<\infty\Big\}\\
&H_2(p)=\Big\{a=(a_n)\in w:  \displaystyle\sum_{n=0}^{\infty}\Big|\displaystyle\sum_{k=0}^{\infty} \displaystyle\sum_{j=0}^{n-k} (-1)^j \frac{D_j^{(s)}}{t_{j+k}}r_ka_n \Big|<\infty\Big\}\\
&H_3(p)=\bigcap_{L\in \mathbb{N}}\Big\{a=(a_n)\in w:  \displaystyle\sup_{F}\displaystyle\sum_{n=0}^{\infty}\Big|
\displaystyle \sum_{k\in F}\displaystyle\sum_{j=0}^{n-k} (-1)^j \frac{D_j^{(s)}}{t_{j+k}}r_ka_n L^{\frac{1}{p_k}}\Big|<\infty\Big\}\\
& H_4(p)=\Big\{a=(a_n)\in w: \displaystyle\sup_{F}\displaystyle\sup_{k\in \mathbb{N}_0}\Big|\displaystyle\sum_{n\in F}
\displaystyle\sum_{j=0}^{n-k} (-1)^j \frac{D_j^{(s)}}{t_{j+k}}r_ka_n\Big|^{p_k}<\infty \Big\}\\
& H_5(p)=\bigcup_{L\in \mathbb{N}}\Big\{a=(a_n)\in w:
\displaystyle\sup_{F}\displaystyle\sum_{k=0}^{\infty}\Big|\displaystyle\sum_{n \in
F}\displaystyle\sum_{j=0}^{n-k} (-1)^j
\frac{D_j^{(s)}}{t_{j+k}}r_ka_nL^{-1}\Big|^{p_k'}<\infty\Big\}.
\end{align*}

\begin{thm}
$(a)$ If $ p_k > 1$, then $[l(r, s, t, p ; \Delta)]^\alpha=H_5(p)$
 and ~$[l(r, s, t, p ; \Delta)]^\alpha=  H_4(p)$ for $0 < p_k \leq
1$.\\
$(b)$ For $0 < p_k \leq H < \infty$, then \\
$(i)$ $[c_0(r, s, t, p ; \Delta)]^\alpha= H_1(p)$,\\
$(ii)$ $ [c(r, s, t, p ; \Delta)]^\alpha= H_1(p)\cap H_2(p)$,\\
$(iii)$ $[l_\infty(r, s, t, p ; \Delta)]^\alpha= H_3(p)$.
\end{thm}
\begin{proof}$(a)$ Let $p_k>1$ $\forall k$, $a=(a_{n}) \in w$, $x\in l(r, s, t, p; \Delta)$ and $y\in l(p)$. Then for each $n$, we have
$$ a_{n}x_n = \sum_{k=0}^{n}\sum_{j=0}^{n-k} (-1)^{j} \frac{D_{j}^{(s)}}{t_{j+k}}r_{k}a_{n}y_{k} =(Cy)_{n},$$
where the matrix $C=(c_{nk})_{n, k}$ is defined as
$$
c_{nk} = \left\{
\begin{array}{ll}
 \displaystyle\sum_{j=0}^{n-k} (-1)^j \frac{D_j^{(s)}}{t_{j+k}}r_ka_n& \mbox{if}~~   0\leq k \leq n \\
0 & \mbox{if} ~~ k>n
\end{array}
\right.
$$
and $x_n$ is given by $(4.3)$. Thus for each $x \in l(r, s, t, p;
\Delta)$, $(a_nx_{n})_{n} \in l_{1}$ if and only if $(Cy)_{n} \in
l_{1}$ where $y \in l(p)$. Therefore $a=(a_{n}) \in [l(r, s, t, p;
\Delta)]^{\alpha}$ if and only if $C \in (l(p), l_1)$. By using
Lemma \ref{lem1} $(i)$, we have
$$\displaystyle\sup_{F}\displaystyle\sum_{k=0}^{\infty}\Big|\displaystyle\sum_{n \in
F}\displaystyle\sum_{j=0}^{n-k} (-1)^j
\frac{D_j^{(s)}}{t_{j+k}}r_ka_nL^{-1}\Big|^{p_k'}<\infty.$$
Hence $[l(r, s, t, p; \Delta)]^{\alpha}= H_5(p).$\\
If $0<p_k\leq1$ $\forall k$, then by using
Lemma \ref{lem1} $(i)$, we have
$$\displaystyle\sup_{F}\displaystyle\sup_{k\in \mathbb{N}_0}\Big|\displaystyle\sum_{n\in F}
\displaystyle\sum_{j=0}^{n-k} (-1)^j \frac{D_j^{(s)}}{t_{j+k}}r_ka_n\Big|^{p_k}<\infty.$$
Thus $[l(r, s, t, p; \Delta)]^{\alpha}= H_4(p)$.\\
$(b)$ In a similar way, using Lemma \ref{lem2}$(i)$, Lemma \ref{lem3}$(i)$ and Lemma \ref{lem4}$(i)$, we obtain $[c_0(r, s, t, p ; \Delta)]^\alpha= H_1(p)$,
$[c(r, s, t, p ; \Delta)]^\alpha= H_1(p)\cap H_2(p)$ and $[l_\infty(r, s, t, p ; \Delta)]^\alpha= H_3(p)$ respectively.
\end{proof}

To compute the $\gamma$-duals of the sequence spaces $X(r,s,t, p ; \Delta)$, we first consider the following sets:
\begin{align*}
&\Gamma_1(p)= \bigcup_{L\in \mathbb{N}}\Big\{ a=(a_k)\in w: \displaystyle\sup_{n\in \mathbb{N}_0}\displaystyle\sum_{k=0}^{\infty} |e_{nk}| L^{\frac{-1}{p_k}}<\infty\Big\}\\
&\Gamma_2(p)= \Big\{ a=(a_k)\in w: \displaystyle\sup_{n\in \mathbb{N}_0}\Big|\displaystyle\sum_{k=0}^{\infty} e_{nk}\Big|<\infty\Big\}\\
&\Gamma_3(p)= \bigcap_{L\in \mathbb{N}}\Big\{ a=(a_k)\in w: \displaystyle\sup_{n\in \mathbb{N}_0}\displaystyle\sum_{k=0}^{\infty}|e_{nk}| L^{\frac{1}{p_k}}<\infty\Big\}\\
&\Gamma_4(p)= \bigcup_{L\in \mathbb{N}}\Big\{ a=(a_k)\in w: \displaystyle\sup_{n\in \mathbb{N}_0}\displaystyle\sup_{k \in \mathbb{N}_0 }| e_{nk}L^{-1}|^{p_k}<\infty\Big\}\\
&\Gamma_5(p)= \bigcup_{L\in \mathbb{N}}\Big\{ a=(a_k)\in w: \displaystyle\sup_{n\in \mathbb{N}_0}\displaystyle\sum_{k=0}^{\infty}|e_{nk}L^{-1}|^{p_k'}<\infty\Big\},
\end{align*}
where the matrix $E=(e_{nk})$ is defined as
\begin{equation}{\label{eq2}}
\displaystyle E=(e_{nk})= \left\{
\begin{array}{ll}
   \displaystyle r_k \bigg[\frac{a_{k}}{s_0t_k} + \Big(\frac{D_{0}^{(s)}}{t_k}- \frac{D_{1}^{(s)}}{t_{k+1}}
   \Big)\sum_{j=k+1}^{n}a_{j}+ \sum_{l=k+2}^{n}(-1)^{l-k} \frac{D_{l-k}^{(s)}}{t_{l}}\Big(\sum_{j=l}^{n}a_{j}\Big)\bigg] & \quad 0\leq k \leq n,\\
    0 & \quad k > n.
\end{array}\right.
\end{equation}
Note: We mean $\displaystyle \sum_{n}^{k} = 0$ if $n > k$.
\begin{thm}
$(a)$ If $ p_k > 1$, then $[l(r, s, t, p ; \Delta)]^\gamma=  \Gamma_5(p)$
~and~ $[l(r, s, t, p ; \Delta)]^\gamma=  \Gamma_4(p)$ if $0 < p_k \leq 1$.\\
$(b)$ If $0 < p_k \leq H < \infty$, then \\
$(i)$ $[c_0(r, s, t, p ; \Delta)]^\gamma= \Gamma_1(p)$,\\
$(ii)$ $ [c(r, s, t, p ; \Delta)]^\gamma= \Gamma_1(p)\cap \Gamma_2(p)$,\\
$(iii)$ $[l_\infty(r, s, t, p ; \Delta)]^\gamma= \Gamma_3(p)$.
\end{thm}
\begin{proof}
$(a)$ Let $p_k > 1$ $\forall k$, $a
=(a_k)\in w$, $x \in l(r,s,t, p ; \Delta)$ and $y \in l(p)$. Then by
using (4.3), we have

\begin{align*}
\displaystyle\sum_{k=0}^{n}a_k x_k & =\sum_{k=0}^{n}\sum_{j=0}^{k}\sum_{l=0}^{k-j}(-1)^{l} \frac{D_{l}^{(s)}r_{j}y_{j}a_{k}}{t_{l+j}} \\
& = \sum_{k=0}^{n-1}\sum_{j=0}^{k}\sum_{l=0}^{k-j}(-1)^{l} \frac{D_{l}^{(s)}r_{j}y_{j}a_{k}}{t_{l+j}}+ \sum_{j=0}^{n}\sum_{l=0}^{n-j}(-1)^{l} \frac{D_{l}^{(s)}r_{j}y_{j}a_{n}}{t_{l+j}}\\
& =  \bigg[\frac{D_{0}^{(s)}}{t_{0}}a_0 + \Big( \frac{D_{0}^{(s)}}{t_{0}}- \frac{D_{1}^{(s)}}{t_{1}} \Big) \sum_{j=1}^{n}a_j + \sum_{l=2}^{n}(-1)^{l} \frac{D_{l}^{(s)}}{t_{l}} \Big( \sum_{j=l}^{n}a_j\Big) \bigg]r_0 y_0 \nonumber\\ &  ~~ +
\bigg[\frac{D_{0}^{(s)}}{t_{1}}a_1 + \Big( \frac{D_{0}^{(s)}}{t_{1}}- \frac{D_{1}^{(s)}}{t_{2}} \Big) \sum_{j=2}^{n}a_j + \sum_{l=3}^{n}(-1)^{l-1} \frac{D_{l-1}^{(s)}}{t_{l}} \Big( \sum_{j=l}^{n}a_j\Big)\bigg]r_1 y_1 + \cdots + \frac{r_n a_n}{t_n} D_{0}^{(s)}y_{n}\nonumber\\
\end{align*}
\begin{align}
& = \sum_{k=0}^{n}r_k \bigg[\frac{a_{k}}{s_0t_k} +
\Big(\frac{D_{0}^{(s)}}{t_k}- \frac{D_{1}^{(s)}}{t_{k+1}} \Big)
\sum_{j=k+1}^{n}a_{j}+ \sum_{l=k+2}^{n}(-1)^{l-k}
\frac{D_{l-k}^{(s)}}{t_{l}}\Big(\sum_{j=l}^{n}a_{j}\Big) \bigg]y_{k}\nonumber\\
&= (Ey)_{n},
\end{align}
where the matrix $E$ is defined in (\ref{eq2}). \\Thus $a \in \big[l(r,s,t, p ; \Delta)\big]^{\gamma}$ if and only if $ax=(a_kx_k)\in bs$, where
$x\in l(r,s,t, p ; \Delta)$ if and only if
$\Big(\displaystyle\sum_{k=0}^{n}a_k x_k \Big)_{n}\in
l_{\infty}$, i.e., $(Ey)_{n} \in l_{\infty}$, where $y\in l(p)$. Hence by using Lemma \ref{lem1}$(iv)$ with $q_n=1$ $\forall n$, we have
$$\displaystyle\sup_{n\in \mathbb{N}_0}\displaystyle\sum_{k=0}^{\infty}|e_{nk}L^{-1}|^{p_k'}<\infty, \mbox{~for some~} L\in \mathbb{N}. $$
 Hence $\big[l(r,s,t, p ; \Delta)\big]^{\gamma} = \Gamma_{5}(p).$\\
If $0<p_k\leq1$ $\forall k$, then using
Lemma \ref{lem1} $(iv)$, we have
$$\displaystyle\sup_{n\in \mathbb{N}_0}\displaystyle\sup_{k \in \mathbb{N}_0 }| e_{nk}L^{-1}|^{p_k}<\infty \mbox{~for some~} L\in \mathbb{N}.$$
Thus $[l(r, s, t, p; \Delta)]^{\gamma}= \Gamma_4(p)$.\\
$(b)$ In a similar way, using Lemma \ref{lem2}$(iv)$, Lemma \ref{lem3}$(iv)$ and Lemma \ref{lem4}$(iv)$, we obtain $[c_0(r, s, t, p ; \Delta)]^\gamma= \Gamma_1(p)$,
$[c(r, s, t, p ; \Delta)]^\gamma= \Gamma_1(p)\cap \Gamma_2(p)$ and $[l_\infty(r, s, t, p ; \Delta)]^\gamma=\Gamma_3(p)$ respectively.
\end{proof}

To find the $\beta$-duals of
$X(r,s,t,p; \Delta)$, we define the following sets:

\begin{align*}
&B_{1}= \Big\{ a=(a_n)\in w: \displaystyle\sum_{j =k+1}^{\infty}a_j ~~{\rm exists~ for~ all}~ k \Big \},\\
&B_{2}= \Big\{ a=(a_n)\in w: \displaystyle\sum_{j =k+2}^{\infty}(-1)^{j-k} \frac{D_{j-k}^{(s)}}{t_{j}}\sum_{l =j}^{\infty}a_l ~~{\rm exists~ for~ all}~ k \Big \},\\
&B_{3}= \Big\{ a=(a_n)\in w: \Big(\frac{r_k a_k}{t_k}\Big) \in l_{\infty}(p) \Big \},\\
&B_{4}= \bigcup_{L\in \mathbb{N}}\Big\{ a=(a_n)\in w: \displaystyle\sup_{n\in \mathbb{N}_0}\sum_{k=0}^{\infty}\Big|e_{nk}L^{-1}\Big|^{p_{k}{'}}<\infty\Big \}, \\
&B_{5}= \Big\{ a=(a_n)\in w: \displaystyle\sup_{n, k\in \mathbb{N}_0
}|e_{nk}|^{p_k}< \infty \Big \},\\
&B_{6}= \Big\{ a=(a_n)\in w: \exists (\alpha_k)~\displaystyle\lim_{n \rightarrow \infty}e_{nk}= \alpha_k ~ \forall ~k \Big \},\\
&B_{7}= \bigcap_{L\in \mathbb{N}}\Big\{ a=(a_n)\in w: \exists (\alpha_k)~ \displaystyle\sup_{n, k\in \mathbb{N}_0
}\Big(|e_{nk} - \alpha_k | L \Big)^{p_k}< \infty \Big \}, \\
&B_{8}= \bigcap_{L\in \mathbb{N}}\Big\{ a=(a_n)\in w: \exists (\alpha_k)~\displaystyle\sup_{n\in \mathbb{N}_0
}\sum_{k=0}^{\infty}\Big(|e_{nk} - \alpha_k | L \Big)^{p_k'}< \infty \Big \},
\end{align*}
\begin{align*}
&B_{9}= \bigcup_{L\in \mathbb{N}}\Big\{ a=(a_n)\in w: \exists (\alpha_k)~ \displaystyle\sup_{n\in \mathbb{N}_0
}\sum_{k=0}^{\infty}\Big|e_{nk} - \alpha_k \Big| L ^\frac{-1}{p_k}< \infty
\Big \},\\
&B_{10}= \bigcup_{L\in \mathbb{N}}\Big\{ a=(a_n)\in w:\displaystyle\sup_{n\in \mathbb{N}_0
}\sum_{k=0}^{\infty}\Big|e_{nk}\Big| L^\frac{-1}{p_k}< \infty \Big \},\\
&B_{11}= \Big\{ a=(a_n)\in w: \exists \alpha~\displaystyle\lim_{n
}\Big|\sum_{k=0}^{\infty}e_{nk}-\alpha\Big| =0 \Big \},\\
&B_{12}= \bigcap_{L\in \mathbb{N}}\Big\{ a=(a_n)\in w: \displaystyle\sup_{n\in \mathbb{N}_0
}\sum_{k=0}^{\infty}\Big|e_{nk}\Big|L^\frac{1}{p_k} <\infty\Big \},\\
&B_{13}= \bigcap_{L\in \mathbb{N}}\Big\{ a=(a_n)\in w: \exists (\alpha_k)~\displaystyle\lim_{n
}\sum_{k=0}^{\infty}|e_{nk}-\alpha_k|L^\frac{1}{p_k}=0\Big \}.
\end{align*}
\begin{thm}
$(a)$ If $p_k>1$ for all $k$, then $[l(r,s, t, p; \Delta)]^{\beta} = B_1 \bigcap B_2 \bigcap B_3
\bigcap B_4 \bigcap B_6\bigcap B_8$ and if $0< p_k \leq 1 $ for all $k$, then $[l(r,s, t, p; \Delta)]^{\beta} = B_1 \bigcap B_2
\bigcap B_3 \bigcap B_5\bigcap B_6 \bigcap B_7$.
\\
$(b)$ Let $p_k>0$ for all $k$. Then\\
$(i)$ $[c_0(r,s, t, p; \Delta)]^{\beta} = B_1 \bigcap B_2 \bigcap B_3
\bigcap B_6 \bigcap B_9\bigcap B_{10}$.\\
$(ii)$ $[c(r,s, t, p; \Delta)]^{\beta}= B_1 \bigcap B_2 \bigcap B_3\bigcap B_6\bigcap B_{9}\bigcap B_{10} \bigcap B_{11}$.\\
$(iii)$ $[l_{\infty}(r,s, t, p; \Delta)]^{\beta}= B_1 \bigcap B_2 \bigcap B_3\bigcap B_{12}\bigcap B_{13}$.

\end{thm}
\begin{proof}
$(a)$ Let $p_k>1$ for all $k$. We have from $(4.31)$
$$\sum\limits_{k=0}^{n}a_k x_k = (Ey)_n,$$ where the matrix $E$ is
defined in (\ref{eq2}).
Thus $a \in \big[l(r,s,t,p; \Delta)\big]^{\beta}$
if and only if $ax=(a_kx_k)\in cs$ where $x \in l(r,s,t,p; \Delta)$
if and only if $(Ey)_{n} \in c$ where $y\in l(p)$, i.e., $E \in (l(p),
c)$. Hence by Lemma \ref{lem1}$(iii)$ with $q_n=1$ $\forall n$, we have
\begin{align*}
& \exists L\in \mathbb{N} ~\displaystyle\sup_{n\in \mathbb{N}_0}\sum_{k=0}^{\infty}\Big|e_{nk}L^{-1}\Big|^{p_{k}{'}}<\infty,\\
& \exists (\alpha_k)~\displaystyle\lim_{n \rightarrow \infty}e_{nk}= \alpha_k  ~\mbox{for all} ~k, \\
& \exists (\alpha_k)~\displaystyle\sup_{n\in \mathbb{N}_0}\sum_{k=0}^{\infty}\Big(|e_{nk} - \alpha_k| L \Big)^{p_k'}< \infty.
\end{align*}
Therefore $[l(r,s, t, p; \Delta)]^{\beta} = B_1
\bigcap B_2 \bigcap B_3 \bigcap B_4 \bigcap B_6\bigcap B_8.$\\
If $0<p_k\leq1$ $\forall k$, then using Lemma \ref{lem1}$(iii)$ with $q_n=1$, $\forall n$, we have
\begin{align*}
& \displaystyle\sup_{n, k\in \mathbb{N}_0}|e_{nk}|^{p_k}< \infty,\\
& \exists (\alpha_k)~\displaystyle\lim_{n \rightarrow \infty}e_{nk}= \alpha_k  ~\mbox{for all} ~k, \\
& \forall L\in \mathbb{N} ~\exists (\alpha_k)~ \displaystyle\sup_{n, k\in \mathbb{N}_0
}\Big(|e_{nk} - \alpha_k | L \Big)^{p_k}< \infty.
\end{align*}
Thus $[l(r,s, t, p; \Delta)]^{\beta} = B_1
\bigcap B_2 \bigcap B_3 \bigcap B_5 \bigcap B_6\bigcap B_7$.\\
$(b)$ In a similar way, using Lemma \ref{lem2}$(iii)$, Lemma \ref{lem3}$(iii)$ and Lemma \ref{lem4}$(iii)$, we can obtain the $\beta$-duals of $c_0(r,s, t, p; \Delta)$, $c(r,s, t, p; \Delta)$ and $l_{\infty}(r,s, t, p; \Delta)$ respectively.

\end{proof}

\subsection{Matrix mappings}
\begin{thm}Let $\tilde{E}=(\tilde{e}_{nk})$ be the matrix which is same as the matrix
${E}=({e}_{nk})$ defined in (\ref{eq2}), where $a_{j}$ is replaced by $a_{nj}$ and $a_k$ by $a_{nk}$. \\
$(a)$ Let
$p_k>1$ for all $k$, then $A \in (l(r,s,
t, p; \Delta), l_{\infty})$ if and only if there exists $L\in \mathbb{N}$ such
that
    $$ \displaystyle\sup_{n}\sum_{k}\Big|\tilde{e}_{nk}L^{-1}\Big|^{p_{k}^{'}} < \infty \mbox{~~and~}(a_{nk})_{k}
     \in B_1 \bigcap B_2 \bigcap B_3 \bigcap B_4 \bigcap B_6\bigcap B_8.$$
$(b)$ Let $0< p_k \leq 1 $ for all $k$. Then $A \in
(l(r,s, t, p; \Delta), l_{\infty})$
 if and only if  $$ \displaystyle\sup_{n, k\in \mathbb{N}_0}\Big|\tilde{e}_{nk}\Big|^{p_{k}} < \infty \mbox{~~and~}(a_{nk})_{k}
   \in B_1 \bigcap B_2 \bigcap B_3 \bigcap
B_5\bigcap B_6 \bigcap B_7.$$
\end{thm}
\begin{proof}
$(a)$ Let $p_k>1$ for all $k$. Since
$(a_{nk})_{k}\in \big[l(r,s,t,p;
\Delta)\big]^{\beta}$ for each fixed $n$, $Ax$ exists for all $x\in l(r,s,t,p;
\Delta) $. Now for each $n$, we have
\begin{align*}
\sum\limits_{k=0}^{m}a_{nk} x_k &=
\sum\limits_{k=0}^{m}\displaystyle r_k
\Big[\frac{a_{nk}}{s_0t_k} + \Big(\frac{D_{0}^{(s)}}{t_k}-
\frac{D_{1}^{(s)}}{t_{k+1}}
   \Big)\sum_{j=k+1}^{n}a_{nj}+ \sum_{j=k+2}^{n}(-1)^{j-k} \frac{D_{j-k}^{(s)}}{t_{j}}\Big(\sum_{l=j}^{n}a_{nl}\Big)\Big
   ]y_k\\
   & = \sum\limits_{k=0}^{m}\tilde{e}_{nk}y_k,
\end{align*}
 Taking
$m\rightarrow\infty$, we have
$$\sum\limits_{k=0}^{\infty}a_{nk}
x_k=\sum\limits_{k=0}^{\infty}\tilde{e}_{nk} y_k  \mbox{~~~for all~}
n.$$ We know that for any $T>0$ and any
complex numbers $a, b$
\begin{equation} {\label{eq1}}
|ab|\leq T (|aT^{-1}|^{p'} + |b|^p),
\end{equation} where $p>1$ and $\frac{1}{p}+ \frac{1}{p'}=1$. Using (\ref{eq1}), we get
$$\displaystyle \sup_{n}\bigg|\sum\limits_{k=0}^{\infty}a_{nk} x_k\bigg|\leq
\sup_{n}\sum\limits_{k=0}^{\infty}\Big|\tilde{e}_{nk} \Big|\Big|y_k\Big|
\leq  T\bigg[\sup_{n}\sum\limits_{k=0}^{\infty}|\tilde{e}_{nk}T^{-1}|^{{p_k}'} +
\sum\limits_{k=0}^{\infty}|y_{k}|^{{p_k}}\bigg]<\infty.$$
Thus $Ax\in l_\infty$. This proves that $A\in (l(r,s,
t, p; \Delta), l_{\infty})$.\\
Conversely, assume that $A\in (l(r,s, t, p; \Delta), l_{\infty})$
and $p_k>1$ for all $k$. Then $Ax$
exists for each $x\in l(r,s, t, p; \Delta)$, which implies that
$(a_{nk})_{k}\in [l(r,s, t, p;
\Delta)]^\beta$ for each $n$. Thus \\
$(a_{nk})_{k}\in B_1 \bigcap B_2 \bigcap B_3
\bigcap B_4 \bigcap B_6\bigcap B_8$. Also from
$\sum\limits_{k=0}^{\infty}a_{nk}
x_k=\sum\limits_{k=0}^{\infty}\tilde{e}_{nk} y_k $, we have
$\tilde{E}=(\tilde{e}_{nk})\in (l(p), l_\infty)$, i.e., for some natural number $L$,
$\displaystyle\sup_{n\in \mathbb{N}_0}\sum_{k=0}^{\infty}\Big|\tilde{e}_{nk}L^{-1}\Big|^{p_{k}^{'}}
<\infty$. This completes the proof.\\
$(b)$ We omit the proof of this part as it is similar to the previous one.
\end{proof}

\begin{thm}
$(a)$ Let $p_k>1$ for all $k$, then $A \in
(l(r,s, t, p; \Delta), l_{1})$ if and only if
\begin{center}
$ \displaystyle\sup_{F} \sum_{k=0}^{\infty}\Big| \sum_{n \in
F}\tilde{e}_{nk} L^{-1}\Big|^{p^{'}_{k}} < \infty ~~ \mbox{for some}~ L \in
\mathbb{N} \mbox{~~and~} (a_{nk})_{k \in \mathbb{N_{\rm 0}}}  \in B_1 \bigcap B_2 \bigcap B_3 \bigcap B_4\bigcap B_6\bigcap B_8.$
 \end{center}
$(b)$ Let $0< p_k \leq 1 $ for all $k$. Then $ A \in
(l(r,s, t, p; \Delta), l_{1})$
 if and only if
\begin{center}
   $ \displaystyle\sup_{ F} \sup_{k}\Big| \sum_{ n \in F}\tilde{e}_{nk}\Big|^{p_{k}} <
  \infty$
and $ ( a_{nk})_{k} \in B_1 \bigcap B_2
\bigcap B_3 \bigcap B_5\bigcap B_6 \bigcap B_7.$
\end{center}
\end{thm}
\begin{proof}
We omit the proof as it follows in a similar way of Theorem $4.6$.
\end{proof}
\begin{cor}
$(a)$ $A\in (c_0(r, s, t, p; \Delta), c_0(q))$ if and only if $(4.6)$, $(4.17)$ hold with $\tilde{e}_{nk}$ in place of $a_{nk}$ and $(a_{nk})\in \big[c_0(r,s,t,p;
\Delta)\big]^{\beta}$,\\
$(b)$ $A\in (c_0(r, s, t, p; \Delta), c(q))$ if and only if $(4.11)$, $(4.18)$, $(4.19)$ hold with $\tilde{e}_{nk}$ in place of $a_{nk}$ and $(a_{nk})\in \big[c_0(r,s,t,p;
\Delta)\big]^{\beta}$,\\
$(b)$ $A\in (c_0(r, s, t, p; \Delta), l_{\infty}(q))$ if and only if $(4.20)$ holds with $\tilde{e}_{nk}$ in place of $a_{nk}$ and $(a_{nk})\in \big[c_0(r,s,t,p;
\Delta)\big]^{\beta}$.
\end{cor}
\begin{cor}
$(a)$ $A\in (c(r, s, t, p; \Delta), c_0(q))$ if and only if $(4.6)$, $(4.17)$, $(4.22)$ hold with $\tilde{e}_{nk}$ in place of $a_{nk}$ and $(a_{nk})\in \big[c(r,s,t,p;
\Delta)\big]^{\beta}$,\\
$(b)$ $A\in (c(r, s, t, p; \Delta), c(q))$ if and only if $(4.11)$, $(4.18)$, $(4.19)$, $(4.23)$ hold with $\tilde{e}_{nk}$ in place of $a_{nk}$ and $(a_{nk})\in \big[c(r,s,t,p;
\Delta)\big]^{\beta}$,\\
$(b)$ $A\in (c(r, s, t, p; \Delta), l_{\infty}(q))$ if and only if $(4.20)$, $(4.24)$ hold with $\tilde{e}_{nk}$ in place of $a_{nk}$ and $(a_{nk})\in \big[c(r,s,t,p;
\Delta)\big]^{\beta}$.
\end{cor}

\begin{cor}
$(a)$ $A\in (l_{\infty}(r, s, t, p; \Delta), c_0(q))$ if and only if $(4.26)$ holds with $\tilde{e}_{nk}$ in place of $a_{nk}$  and $(a_{nk})\in \big[l_{\infty}(r,s,t,p;
\Delta)\big]^{\beta}$,\\
$(b)$ $A\in (l_{\infty}(r, s, t, p; \Delta), c(q))$ if and only if $(4.27)$, $(4.28)$ hold with $\tilde{e}_{nk}$ in place of $a_{nk}$ and $(a_{nk})\in \big[l_{\infty}(r,s,t,p;
\Delta)\big]^{\beta}$,\\
$(b)$ $A\in (l_{\infty}(r, s, t, p; \Delta), l_{\infty}(q))$ if and only if $(4.29)$ holds with $\tilde{e}_{nk}$ in place of $a_{nk}$ and $(a_{nk})\in \big[l_{\infty}(r,s,t,p;
\Delta)\big]^{\beta}$.
\end{cor}
\section{Kadec-Klee property and rotundity of $l(r,s, t, p; \Delta)$}
In many geometric properties of Banach spaces, Kadec-Klee property and rotundity play an important role in metric fixed point theory. These properties are extensively studied in Orlicz spaces (see \cite{CHE}, \cite{HUDZ}, \cite{WU}) and also studied in difference sequence spaces by Kananthai\cite{KANAN}. In this section, we discuss these properties in the sequence space $l(r,s, t, p; \Delta)$. \\
Throughout the paper, for any Banach space $(Y, \|.\|)$, we denote $S(Y)$ and $B(Y)$ as the unit sphere and closed unit ball respectively.\\
A point $x\in S(Y)$ is called an extreme point if $x=\frac{y+z}{2}$ implies $y=z$ for every $y, z \in S(Y)$. A Banach space $Y$ is said to be rotund (strictly convex) if every point of $S(Y)$ is an extreme point.\\
Let $X$ be a real vector space. A functional $\sigma: X \rightarrow [0,\infty]$ is called a modular if\\
$(i)$ $\sigma(x)=0$ if and only if $x=0$,\\
$(ii)$ $\sigma(-x)= \sigma(x)$,\\
$(iii)$ $\sigma(\alpha x+ \beta y)\leq \sigma(x)+ \sigma(y)$ for all $x, y\in X$ and $\alpha, \beta\geq0$ with $\alpha+ \beta=1$.\\
A modular $\sigma$ is said to be convex if \\
$(iv)$ $\sigma(\alpha x+ \beta y)\leq \alpha\sigma(x)+ \beta\sigma(y)$ for all $x, y\in X$ and $\alpha, \beta\geq0$ with $\alpha+ \beta=1$.\\
For any modular $\sigma$, the modular space $X_\sigma$ is defined by
$$X_\sigma=\{x\in X: \sigma(\lambda x)\rightarrow 0 \mbox{~as~} \lambda\rightarrow 0+\}.$$
We define $X_\sigma^{*}=\{x\in X: \sigma(\lambda x)<\infty \mbox{~for some ~} \lambda>0\}$. It is clear that $X_\sigma\subseteq X_\sigma^{*}$. Orlicz \cite{ORC} prove that if  $\sigma$ is convex then $X_\sigma=X_\sigma^{*}$.\\
 A modular $\sigma$ is said to be\\
$(i)$ Right continuous if $\displaystyle\lim_{\lambda\rightarrow 1+}\sigma(\lambda x)=\sigma(x)$,\\
$(ii)$ Left continuous if $\displaystyle\lim_{\lambda\rightarrow 1-}\sigma(\lambda x)=\sigma(x)$,\\
$(iii)$ Continuous if it is both left and right continuous.\\
A modular $\sigma$ is said to satisfy $\Delta_2$-condition \cite{CUI5}, denoted by $\sigma\in \Delta_2$ if for any $\epsilon>0$, there exist constants $K\geq2$ and $a>0$ such that $\sigma(2x)\leq K \sigma(x)+ \epsilon$ for all $x\in X_\sigma$ with $\sigma(x)\leq a$.\\
If $\sigma$ satisfies $\Delta_2$-condition for any $a>0$ with $K\geq2$ dependent on $a$, we say that $\sigma$ satisfies strong $\Delta_2$-condition, denoted by $\sigma\in \Delta_2^s$ \cite{CUI5}.\\
Let $p_n>1$ for all $n\in \mathbb{N_{\rm 0}}$. Then for $x \in
l(r,s, t, p; \Delta)$, we define
$$ \sigma_{p}(x) = \sum_{n=0}^{\infty}\bigg |\frac{1}{r_{n}}\sum_{k=0}^{n} s_{n-k}t_{k}\Delta x_{k}\bigg|^{p_n} \qquad (n \in \mathbb{N_{\rm 0}}).$$
By the convexity of the function $t \longmapsto |t|^{p_n}$ for each
$n \in \mathbb{N_{\rm 0}}$, we have $\sigma_{p}$ is a convex modular on
$l(r,s, t, p; \Delta)$.\\
We consider $l(r,s, t, p; \Delta)$ equipped with the so called Luxemburg norm given by
$$ \|x \| = \inf \Big\{ c > 0: \sigma_p\Big(\frac{x}{c}\Big) \leq 1 \Big\}.$$
A normed sequence space $X$ is said to be $K$-space if each
coordinate mapping $P_k$ defined by $P_k(x)= x_k$ is continuous for each $k\in \mathbb{N_{\rm 0}}$. If
$X$ is a Banach space as well as $K$-space, it is called a BK space.
Let $p_k \geq 1$ $\forall k \in  \mathbb{N_{\rm 0}}$ and $M =
\displaystyle\sup_{k} p_k$. It is easy to verify that $\sigma_{p}$ satisfies the strong $\Delta_2$-condition, i.e., $\sigma_p\in \Delta_2^s$.
\begin{pro}\label{pro2}
For $x \in l(r,s, t, p; \Delta)$, the modular $\sigma_{p}$ on $l(r,s, t, p; \Delta)$ satisfies the following:\\
$(i)$ if $0< \alpha \leq 1$, then $\alpha^{M}\sigma_p(\frac{x}{\alpha}) \leq \sigma_p({x})$ and $\sigma_p(\alpha{x})\leq \sigma_p({x})$.\\
$(ii)$ if $\alpha \geq 1$, then $\sigma_p({x}) \leq \alpha^{M}\sigma_p(\frac{x}{\alpha})  $.\\
$(iii)$ if $\alpha \geq 1$, then $\sigma_p({x}) \leq \alpha\sigma_p({x}) \leq \sigma_p(\alpha{x})$.
\end{pro}
\begin{proof}
$(i)$ We have $$\sigma_p\Big(\frac{x}{\alpha}\Big)=\sum_{n=0}^{\infty}\bigg |\frac{1}{\alpha r_{n}}\sum_{k=0}^{n} s_{n-k}t_{k}\Delta x_{k}\bigg|^{p_n}\leq \frac{1}{\alpha^M} \sum_{n=0}^{\infty}\bigg |\frac{1}{r_{n}}\sum_{k=0}^{n} s_{n-k}t_{k}\Delta x_{k}\bigg|^{p_n}=\frac{1}{\alpha^M}\sigma_p({x}),$$
i.e., $\alpha^{M}\sigma_p(\frac{x}{\alpha}) \leq \sigma_p({x})$ and using convexity of $\sigma_p$, we have $\sigma_p(\alpha{x})\leq \sigma_p({x})$ for $0< \alpha \leq 1$.\\
Statements $(ii)$ and $(iii)$ can be proved in a similar way. So, we omit the details.
\end{proof}
\begin{pro}\label{pro1}
The modular $\sigma_p$ is continuous.
\end{pro}
\begin{proof}
Let $\lambda > 1$. From Proposition \ref{pro2}, we have
$$\sigma_p({x}) \leq \lambda \sigma_p({x}) \leq \sigma_p(\lambda{x}) \leq \lambda^{M}\sigma_p({x}). $$
Taking $ \lambda \rightarrow 1+$, we obtain $\displaystyle\lim_{\lambda \rightarrow 1+}\sigma_p(\lambda{x}) = \sigma_p({x}).$ So $\sigma_p$
is right continuous.\\
If $0< \lambda <1$, then we have $ \lambda^{M}\sigma_p({x}) \leq
\sigma_p(\lambda{x}) \leq \lambda \sigma_p({x})$. Taking $
\lambda \rightarrow 1-$, we obtain $\displaystyle\lim_{\lambda
\rightarrow 1-}\sigma_p(\lambda{x}) = \sigma_p({x}).$ So $\sigma_p$
is left continuous. Thus $\sigma_p$ is continuous.
\end{proof}
 Now we give some relationship between norm and modular.
\begin{pro}{\label{pro4}}
For any $x\in l(r,s, t, p; \Delta)$, we have\\
$(i)$ if $\|x\|<1$ then $\sigma_p(x)\leq \|x\|$,\\
$(ii)$ if $\|x\|>1$ then $\sigma_p(x)\geq \|x\|$,\\
$(iii)$ $\|x\|=1$ if and only if $\sigma_p(x)=1$,\\
$(iv)$ $\|x\|<1$ if and only if $\sigma_p(x)<1$,\\
$(v)$ $\|x\|>1$ if and only if $\sigma_p(x)>1$,\\
$(vi)$ if $0<\alpha<1$ and $\|x\|>\alpha$ then $\sigma_p(x)>\alpha^M$,\\
$(vii)$ if $\alpha\geq1$ and $\|x\|<\alpha$ then
$\sigma_p(x)<\alpha^M$.
\end{pro}
\begin{proof}
$(i)$ Suppose $\|x\|<1$. Let $u$ be a positive number such that $\|x\|<u<1$. Then by the definition of norm $\|.\|$, we have $\sigma_p\Big(\frac{x}{u}\Big)\leq 1$. Using convexity of $\sigma_p$, we have $\sigma_p(x)=\sigma_p\Big(u\frac{x}{u}\Big)<u\sigma_p\Big(\frac{x}{u}\Big)\leq u$. Since $u$ is arbitrary, this implies that $\sigma_p(x)\leq \|x\|$.\\
$(ii)$ Let $u$ be a positive number such that $\|x\|> u >1$. Then $\sigma_p\Big(\frac{x}{u}\Big)> 1$ and $1<\sigma_p\Big(\frac{x}{u}\Big)<\frac{1}{u}\sigma_p(x)$, i.e., $\sigma_p(x)>u$. Taking $u\rightarrow \|x\|_{-}$, we obtain $\sigma_p(x)\geq \|x\|$.\\
$(iii)$ Since $\sigma_p\in \Delta_2^s$, so the proof follows from Corollary $2.2$ in \cite{CUI5} and Proposition \ref{pro1}.\\
$(iv)$ and $(v)$ follows from $(i)$ and $(iii)$.\\
$(vi)$ and $(vii)$ follows from Proposition \ref{pro2}$(i)$ and $(ii)$.
\end{proof}
\begin{pro}\label{pro3}
For any $(x^m)$ be a sequence of elements of $l(r,s, t, p; \Delta)$.\\
$(i)$ If $\|x^m\|\rightarrow1$ then $\sigma_p(x^m)\rightarrow 1$ as $m\rightarrow\infty$,\\
$(ii)$ If $\|x^m\|\rightarrow0$ if and only if
$\sigma_p(x^m)\rightarrow 0$ as $m\rightarrow\infty$.
\end{pro}
\begin{proof}
$(i)$ Suppose that $\|x^m\|\rightarrow1$ as $m\rightarrow\infty$. Then
for every $\epsilon\in(0,1)$ there exists $N\in \mathbb{N_{\rm 0}}$
such that $1-\epsilon<\|x^m\|< 1+\epsilon$ for all $m\geq N$. Thus
by Proposition \ref{pro4} $(vi)$ and $(vii)$, we have
$(1-\epsilon)^M<\sigma_p(x^m)< (1+\epsilon)^M$ for all $m\geq N$.
Hence $\sigma_p(x^m)\rightarrow 1$ as $m\rightarrow\infty$.\\
$(ii)$ Since  $\sigma_p\in \Delta_2^s$, so the proof follows from Lemma $2.3$ in \cite{CUI5}.
\end{proof}
\begin{lem}
The space $l(r, s, t, p ; \Delta)$ is a $BK$ space.
\end{lem}
\begin{proof}
Since the space $l(r, s, t, p ; \Delta)$ equipped with the Luxemberg norm
$\|.\|$ is a Banach space, so it is enough to prove that $l(r, s, t, p ; \Delta)$ is a
$K$-space. Suppose $(x^m)\in l(r, s, t, p; \Delta)$ such that
$x^m\rightarrow0$ as $m\rightarrow\infty$. By Proposition \ref{pro3}(ii), we
have $\sigma_p(x^m)\rightarrow0$ as $m\rightarrow\infty$. This
implies that
\begin{center}
$\Big|\frac{1}{r_n}\displaystyle\sum_{k=0}^{n}s_{n-k}t_k\Delta
x_k^{m}\Big|^{p_n}\rightarrow 0$ \mbox{~as~} $m\rightarrow\infty$
\mbox{~and for each~} $n\in \mathbb{N_{\rm 0}}$.
\end{center}By induction, we have $x_k^{{m}}\rightarrow 0$ as
$m\rightarrow\infty$ for each $k\in \mathbb{N}_0$. Hence the coordinate mappings $P_k(x^m)=
x_k^{m}\rightarrow 0$ as $m\rightarrow\infty$ which implies that
$P_k$'s are continuous for each $k$.
\end{proof}

\begin{lem}{\label{lem5}}
Let $x\in l(r, s, t, p ; \Delta)$ and $(x^m)\subseteq l(r, s, t, p ;
\Delta)$. If $\sigma_p(x^m)\rightarrow \sigma_p(x)$ and
$x_k^{m}\rightarrow x_k$ as $m\rightarrow\infty$ for each $k$ then
$x^m\rightarrow x$.
\end{lem}
\begin{proof}
Since $x\in l(r, s, t, p ; \Delta)$ i.e., $\sigma_p(x)<\infty$, so
for a given $\epsilon>0$ there exists $n_0\in \mathbb{N}$
such that
\begin{equation}
\displaystyle
\sum_{n=n_0+1}^{\infty}\Big|\frac{1}{r_n}\displaystyle\sum_{k=0}^{n}s_{n-k}t_k\Delta
x_k\Big|^{p_n}< \frac{\epsilon}{3}\frac{1}{2^{M+1}}.
\end{equation}
Again since $\sigma_p(x^m)\rightarrow \sigma_p(x)$ and
$x_k^{m}\rightarrow x_k$ as $m\rightarrow\infty$ for each $k$, so
there exists $m_0, n_0\in \mathbb{N}$ such that for $m\geq
m_0$
\begin{equation}
\sigma_p(x^m)-\bigg(\displaystyle
\sum_{n=0}^{n_0}\Big|\frac{1}{r_n}\displaystyle\sum_{k=0}^{n}s_{n-k}t_k\Delta
x_k^{m}\Big|^{p_n}\bigg)< \sigma_p(x)-\bigg(\displaystyle
\sum_{n=0}^{n_0}\Big|\frac{1}{r_n}\displaystyle\sum_{k=0}^{n}s_{n-k}t_k\Delta
x_k\Big|^{p_n}\bigg) + \frac{\epsilon}{3}\frac{1}{2^{M}}
\end{equation}
and
\begin{equation}
\bigg(\displaystyle
\sum_{n=0}^{n_0}\Big|\frac{1}{r_n}\displaystyle\sum_{k=0}^{n}s_{n-k}t_k(\Delta
x_k^{m}-\Delta x_k)\Big|^{p_n}\bigg)< \frac{\epsilon}{3}.
\end{equation}
Thus for $m\geq m_0$, we have
\begin{align*}
\sigma_p(x^m-x)&=\displaystyle
\sum_{n=0}^{\infty}\Big|\frac{1}{r_n}\displaystyle\sum_{k=0}^{n}s_{n-k}t_k(\Delta
x_k^{m}-\Delta x_k)\Big|^{p_n}\\
&= \displaystyle
\sum_{n=0}^{n_0}\Big|\frac{1}{r_n}\displaystyle\sum_{k=0}^{n}s_{n-k}t_k(\Delta
x_k^{m}-\Delta x_k)\Big|^{p_n}+ \displaystyle
\sum_{n=n_0+1}^{\infty}\Big|\frac{1}{r_n}\displaystyle\sum_{k=0}^{n}s_{n-k}t_k(\Delta
x_k^{m}-\Delta x_k)\Big|^{p_n}\\
&< \frac{\epsilon}{3} + 2^M\Big\{\displaystyle
\sum_{n=n_0+1}^{\infty}\Big|\frac{1}{r_n}\displaystyle\sum_{k=0}^{n}s_{n-k}t_k\Delta
x_k^{m}\Big|^{p_n}+ \displaystyle
\sum_{n=n_0+1}^{\infty}\Big|\frac{1}{r_n}\displaystyle\sum_{k=0}^{n}s_{n-k}t_k\Delta
x_k\Big|^{p_n} \Big\}\\
\end{align*}
\begin{align*}
&=\frac{\epsilon}{3} + 2^M\Big\{\sigma_p(x^m)-\displaystyle
\sum_{n=0}^{n_0}\Big|\frac{1}{r_n}\displaystyle\sum_{k=0}^{n}s_{n-k}t_k\Delta
x_k^{m}\Big|^{p_n}+ \displaystyle
\sum_{n=n_0+1}^{\infty}\Big|\frac{1}{r_n}\displaystyle\sum_{k=0}^{n}s_{n-k}t_k\Delta
x_k\Big|^{p_n} \Big\}\\
&<\frac{\epsilon}{3} + 2^M\Big\{\sigma_p(x)-\displaystyle
\sum_{n=0}^{n_0}\Big|\frac{1}{r_n}\displaystyle\sum_{k=0}^{n}s_{n-k}t_k\Delta
x_k\Big|^{p_n}+ \frac{\epsilon}{3.2^M }+ \displaystyle
\sum_{n=n_0+1}^{\infty}\Big|\frac{1}{r_n}\displaystyle\sum_{k=0}^{n}s_{n-k}t_k\Delta
x_k\Big|^{p_n} \Big\}\\
& =\frac{\epsilon}{3} + 2^M\Big\{\displaystyle
\sum_{n=n_0+1}^{\infty}\Big|\frac{1}{r_n}\displaystyle\sum_{k=0}^{n}s_{n-k}t_k\Delta
x_k\Big|^{p_n}\Big\}+ \frac{\epsilon}{3} +
2^M.\frac{\epsilon}{3}\frac{1}{2^{M+1}}\\
&<\frac{\epsilon}{3} +
2^M.\frac{\epsilon}{3}\frac{1}{2^{M+1}}+\frac{\epsilon}{3} +
\frac{\epsilon}{6}= \epsilon.
\end{align*}
This shows that $\sigma_p(x^m-x)\rightarrow0$ as
$m\rightarrow\infty$. Therefore by Proposition \ref{pro3}, we have
$x^m\rightarrow x$ in norm.
\end{proof}

\begin{thm}
The space $l(r, s, t, p ; \Delta)$ has the Kadec-Klee property.
\end{thm}
\begin{proof}
Let $x\in S(l(r, s, t, p ; \Delta))$ and $(x^m)$ be a sequence in $l(r, s, t,
p ; \Delta)$ such that $\|x^m\|\rightarrow1$ as $m\rightarrow\infty$
and $x^m\rightarrow x$ weakly as $m\rightarrow\infty$. Since
$\|x\|=1$ so by Proposition \ref{pro4}(iii), we have $\sigma_p(x)=1$. By the continuity of
the functional $\sigma_p$, we have $\sigma_p(x^m)\rightarrow
\sigma_p(x)$ as $m\rightarrow\infty$. It is known that weak
convergence implies the coordinate wise convergence, i.e.,
$x_k^m\rightarrow x_k$ as $m\rightarrow\infty$ for each $k$, so by
the Lemma \ref{lem5}, we obtain $x^m\rightarrow x$ as $m\rightarrow\infty$.
\end{proof}

\begin{thm}
The space $l(r, s, t, p ; \Delta)$ is rotund if $p_n>1$
for each $n$.
\end{thm}

\begin{proof}
Let $x\in S(l(r, s, t, p ; \Delta))$ and $y, z \in B(l(r, s, t, p ;
\Delta))$ with $x=\frac{y+z}{2}$. We have to show that $y=z$. Since
$\sigma_p(x)=1$ and
\begin{center}
$1=\sigma_p(x)= \sigma_p(\frac{y+z}{2})\leq
\frac{1}{2}(\sigma_p(y)+\sigma_p(z))\leq 1$,
\end{center}
we have $\sigma_p(x)=\frac{1}{2}(\sigma_p(y)+\sigma_p(z))$ and
$\sigma_p(y)=1$, $\sigma_p(z)=1$.\\
Therefore, we have
\begin{center}
$\displaystyle
\sum_{n=0}^{\infty}\Big|\frac{1}{r_n}\displaystyle\sum_{k=0}^{n}s_{n-k}t_k\Delta
x_k\Big|^{p_n}=\frac{1}{2}\sum_{n=0}^{\infty}\Big|\frac{1}{r_n}\displaystyle\sum_{k=0}^{n}s_{n-k}t_k\Delta
y_k\Big|^{p_n}+ \frac{1}{2}
\sum_{n=0}^{\infty}\Big|\frac{1}{r_n}\displaystyle\sum_{k=0}^{n}s_{n-k}t_k\Delta
z_k\Big|^{p_n} $.
\end{center}
Since $x=\frac{y+z}{2}$, we have from above
\begin{center}
$\displaystyle
\sum_{n=0}^{\infty}\Big|\frac{1}{r_n}\displaystyle\sum_{k=0}^{n}s_{n-k}t_k\frac{\Delta
y_k+ \Delta
z_k}{2}\Big|^{p_n}=\frac{1}{2}\sum_{n=0}^{\infty}\Big|\frac{1}{r_n}\displaystyle\sum_{k=0}^{n}s_{n-k}t_k\Delta
y_k\Big|^{p_n}+
\frac{1}{2}\sum_{n=0}^{\infty}\Big|\frac{1}{r_n}\displaystyle\sum_{k=0}^{n}s_{n-k}t_k\Delta
z_k\Big|^{p_n} $.
\end{center}By the strict convexity of the function $f(t)=|t|^{p_k},
p_k>1$ for each $k$, from above, we obtain for each $n$
\begin{center}
$\displaystyle
\frac{1}{2r_n}\displaystyle\sum_{k=0}^{n}s_{n-k}t_k\Delta y_k =
\frac{1}{2r_n}\displaystyle\sum_{k=0}^{n}s_{n-k}t_k\Delta z_k$.
\end{center}By induction, we obtain $y_k=z_k$ for each $k\in
\mathbb{N_{\rm{0}}}$, i.e., $y=z$.
\end{proof}

\end{document}